\documentclass[11pt]{article}
\usepackage[utf8]{inputenc}
\usepackage[utf8]{inputenc}
\usepackage{hyperref}
\usepackage{amsmath,amssymb,amsthm}
\usepackage{amsfonts,tikz}
\usepackage[margin=1in]{geometry}
\usepackage{mathtools}
\usetikzlibrary{positioning,automata}
\tikzset{every loop/.style={min distance=10 mm, in=60, out=120, looseness=10}}
\newcommand{\bk}{\backslash}

\newtheorem{lemma}[]{Lemma}

\newtheorem{theorem}[lemma]{Theorem}

\title{A Note on Distinguishing Trees with the Chromatic Symmetric Function}
\author{Logan Crew\footnote{Department of Combinatorics \& Optimization, University of Waterloo, Waterloo, ON, N2L 3E9.\newline  Emails:  lcrew@uwaterloo.ca,  }}
\date{\today}

\begin{document}

\maketitle

\begin{abstract}

For a tree $T$, consider its smallest subtree $T^{\circ}$ containing all vertices of degree at least $3$. Then the remaining edges of $T$ lie on disjoint paths each with one endpoint on $T^{\circ}$. We show that the chromatic symmetric function of $T$ determines the size of $T^{\circ}$, and the multiset of the lengths of these incident paths. In particular, this generalizes a proof of Martin, Morin, and Wagner that the chromatic symmetric function distinguishes spiders.

\end{abstract}

\section{Introduction}

The chromatic symmetric function $X_G$ of a graph $G$ was introduced by Stanley in the 1990s \cite{stanley}, and has since been very well-studied due to its connections to other ares of mathematics, most notably algebraic geometry \cite{unit, categ}, and knot theory \cite{chmutov, noble}. 
One of the central unanswered questions driving research into the chromatic symmetric function is whether it distinguishes non-isomorphic trees. It has been verified computationally that the chromatic symmetric function distinguishes trees of up to $29$ vertices \cite{heil}, and many partial results have been discovered \cite{trees, loebl}. In a notable work, Martin, Morin, and Wagner \cite{martin} showed that the chromatic symmetric function of a tree contains the bivariate subtree polynomial of Chaudhary and Gordon \cite{subtreeorig}, and from this they were able to prove that certain infinite families of trees (spiders and some caterpillars) are distinguished by the chromatic symmetric function.

In this work, we show that the bivariate subtree polynomial provides some more general information, namely the size of the smallest subtree that contains all vertices of degree at least $3$, and the lengths of the paths (which must be the rest of the tree) that protrude from it. In particular, this generalizes the proof from \cite{martin} that spiders are distinguished by the chromatic symmetric function.

\section{Background}

A \emph{graph} $G = (V,E)$ consists of a \emph{vertex set} $V$ and an \emph{edge multiset} $E$ where the elements of $E$ are (unordered) pairs of (not necessarily distinct) elements of $V$. A \emph{simple graph} is a graph $G = (V,E)$ in which $E$ is a set with no repeated elements; all graphs in this paper are simple.  If $\{v_1,v_2\}$ is an edge, we will write it as $v_1v_2 = v_2v_1$.  The vertices $v_1$ and $v_2$ are the \emph{endpoints} of the edge $v_1v_2$. We will use $V(G)$ and $E(G)$ to denote the vertex set and edge set of a graph $G$, respectively.
  
For a vertex $v \in V(G)$, its \emph{degree} $d(v)$ is the number of times $v$ occurs as an endpoint of an edge in $E(G)$. If the vertices of $G$ are labelled $v_1, \dots, v_n$ such that $d(v_1) \geq \dots \geq d(v_n)$, the \emph{degree sequence} of $G$ is the tuple $(d(v_1), \dots, d(v_n))$.

  A \emph{subgraph} of a graph $G$ is a graph $G' = (V',E')$ where $V' \subseteq V$ and $E' \subseteq E|_{V'}$, where $E|_{V'}$ is the set of edges with both endpoints in $V'$.  An \emph{induced subgraph} of $G$ is a graph $G' = (V',E|_{V'})$ with $V' \subseteq V$.

  A \emph{path} in a graph $G$ is a nonempty sequence of edges $v_1v_2$, $v_2v_3$, \dots, $v_{k-1}v_k$ such that $v_i \neq v_j$ for all $i \neq j$. The vertices $v_1$ and $v_k$ are the \emph{endpoints} of the path. A \emph{cycle} in a graph is a nonempty sequence of distinct edges $v_1v_2$, $v_2v_3$, \dots, $v_kv_1$ such that $v_i \neq v_j$ for all $i \neq j$.

  A graph $G$ is \emph{connected} if for every pair of vertices $v_1$ and $v_2$ of $G$ there is a path in $G$ with $v_1$ and $v_2$ as its endpoints. The \emph{connected components} of $G$ are the maximal induced subgraphs of $G$ which are connected.
  
  A \emph{tree} is a connected graph containing no cycles. It is well-known that every tree with $n$ vertices has $n-1$ edges. A connected induced subgraph of a tree with at least one edge is called one of its \emph{subtrees}. A vertex of degree $1$ in a tree is called a \emph{leaf}.
  
  For a tree $T$, its \emph{bivariate subtree polynomial} $S_T(q,r)$, introduced by Chaudhary and Gordon \cite{subtreeorig}, is defined by
  $$
  S_T(q,r) = \sum_{\text{subtrees } S \text{ of } T} q^{|E(S)|}r^{|L(S)|}
  $$
  where $L(S)$ is the set of leaves of $S$ (which need not also be leaves in $T$).
  
  If $G = (V(G),E(G))$ be a graph, a map $\kappa: V(G) \rightarrow \mathbb{Z}^+$ is called a \emph{coloring} of $G$. This coloring is called \emph{proper} if $\kappa(v_1) \neq \kappa(v_2)$ for all $v_1,v_2$ such that there exists an edge $e = v_1v_2$ in $E(G)$. The \emph{chromatic symmetric function} $X_G$ of $G$ is defined as \cite{stanley}
  
  $$X_G(x_1,x_2,\dots) = \sum_{\kappa \text{ proper}} \prod_{v \in V(G)} x_{\kappa(v)}$$ 
  where the sum ranges over all proper colorings $\kappa$ of $G$.

    Martin, Morin, and Wagner \cite{martin} proved the following two results which we will require:

\begin{theorem}\label{thm:subtree}

For any tree $T$, its chromatic symmetric function $X_T$ determines its bivariate subtree polynomial $S_T$.

\end{theorem}

\begin{lemma}\label{lem:path}

The bivariate subtree polynomial of a tree $T$ determines its degree sequence.

\end{lemma}

\section{Main Result}

For a tree $T$ we define its \emph{trunk} $T^{\circ}$ to be the smallest connected induced subgraph that contains all vertices of degree $\geq 3$.  For every leaf $l$, we define its \emph{twig} $\hat{l}$ to be the longest path $P$ in $T$ containing $l$ such that every interior vertex (non-endpoint) of $P$ has degree $2$ in $T$.  We call a path of $T$ a twig if it is a twig for one of its leaves.  Thus, we may view any tree $T$ as the union of its trunk with its twigs, as shown in Figure \ref{fig:trunk_twigs}.

\begin{figure}[hbt]
\begin{center}
  \begin{tikzpicture}[scale=1.5]

  \node[fill=black, circle] at (0, 1)(v1){};
  \node[fill=black, circle] at (1, 1)(v2){};
  \node[fill=black, circle] at (1, 0)(v3){};
  \node[fill=black, circle] at (2, 1)(v4){};
  \node[fill=black, circle] at (2, 2)(v5){};
  \node[fill=black, circle] at (2, 0)(v6){};
  \node[fill=black, circle] at (3, 1)(v7){};
  \node[fill=black, circle] at (4, 1)(v8){};
  \node[fill=black, circle] at (4, 0)(v9){};
  \node[fill=black, circle] at (5, 1)(vt){};
  \node[fill=black, circle] at (6, 1)(ve){};

  \draw[black, dashed] (v1) -- (v2);
  \draw[black, dashed] (v2) -- (v3);
  \draw[black, thick] (v2) -- (v4);
  \draw[black, dashed] (v4) -- (v5);
  \draw[black, dashed] (v4) -- (v6);
  \draw[black, thick] (v4) -- (v7);
  \draw[black, thick] (v7) -- (v8);
  \draw[black, dashed] (v8) -- (v9);
  \draw[black, dashed] (v8) -- (vt);
  \draw[black, dashed] (vt) -- (ve);

  \end{tikzpicture}

\end{center}
\caption{A tree decomposed into its trunk (with thick edges) and twigs (represented by dashed edges).}
\label{fig:trunk_twigs}

\vspace{0.5cm}

\end{figure}
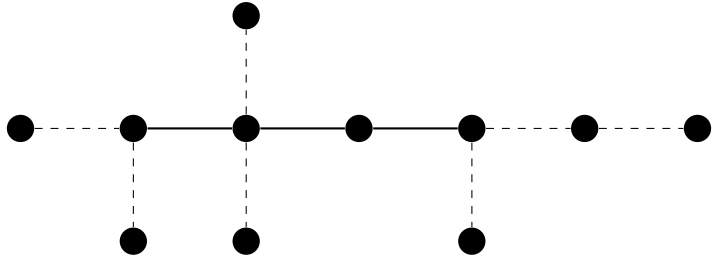

\begin{theorem}

From the bivariate subtree polynomial of a tree $T$ (and thus from its chromatic symmetric function $X_T$ by Theorem \ref{thm:subtree}), we can recover

\begin{itemize}
\item The size of $T^{\circ}$.
\item The lengths of all twigs of $T$.
\end{itemize}

\end{theorem}

\begin{proof}

First, using Lemma \ref{lem:path}, we know the degree sequence of $T$, and thus we know how many leaves (and equivalently twigs) $T$ has; let $a$ be this number. 

 Note that the only way a subtree $S$ of $T$ can have exactly $a$ leaves is if it contains all vertices of degree at least $3$. For suppose otherwise, that we have such a subtree $S$ with $a$ leaves that doesn't contain some particular vertex $x$ of degree $\geq 3$. Then there is exactly one path connecting $x$ to $S$; in addition to the edge starting this path at $x$, there are at least $2$ further edges coming out of $x$, each of which leads to at least one leaf of $T$. Hence the connected component of $T \bk S$ containing $x$ has at least two leaves of $T$, and so $T$ has at least $a+1$ leaves, a contradiction. Thus, a subtree of $T$ with $a$ leaves contains $T^{\circ}$. Furthermore, such a subtree $S$ additionally contains every edge of $T$ adjacent to some vertex of degree at least $3$; for if $S$ does not contain some such edge adjacent to such a vertex $v$, then adding this edge increases the number of leaves of $S$ unless $v$ is itself a leaf of $S$, but then since $v$ has degree at least $3$ in $T$, there are at least two distinct edges we can add to $S$ adjacent to $v$ to increase the number of leaves of $S$.

Thus, every subtree of $T$ with $a$ leaves contains the subtree that includes all of $T^{\circ}$, plus one edge from every twig. Since clearly this subtree itself has exactly $a$ leaves, we conclude that it is the smallest subtree of $T$ having $a$ leaves. Now, we determine from the subtree polynomial the smallest value $v$ such that there exists a subtree $U$ of $T$ with $v$ vertices with $a$ leaves. From the above, it must be the case that $U$ is equal to the union of $T^{\circ}$ with one edge from each twig, and so $|T^{\circ}| = v-a$.

Now, for each $i$ let $t_i$ be the number of twigs of length $i$ in $T$.  We will evaluate the $t_i$ iteratively, starting with $t_1$. To do so, we consider the number of subtrees of $T$ with $v+1$ vertices and $a$ leaves. Clearly all such trees contain $U$, as well as one additional edge from some twig of $T$ with more than one edge. Thus, there will be $a-t_1$ such trees, so since we know $a$, we may compute $t_1$.

Continuing inductively, we assume that for a fixed $k$ we know $t_1, \dots, t_k$, and we therefore know that the remaining $a-\sum_{i=1}^k t_i$ twigs have length at least $k+1$. From the subtree polynomial, we determine the number $S(k)$ of subtrees of $T$ with $v+k+1$ vertices and $a$ leaves. Clearly all such trees contain $U$, as well as $k+1$ additional vertices from twigs. We split into two kinds of trees: those that take all $k+1$ vertices from a single twig, and those that do not. We may count the latter in terms of $t_1, \dots, t_k, a$, since we are guaranteed to be able to take at least $k$ vertices in addition to the initial one in $U$ from the twigs of undetermined length; let this number be $f(t_1, \dots, t_k, a)$. For the former, note that we have already taken one vertex from each twig in $U$, so we may only take an additional $k+1$ vertices from a twig of length at least $k+2$. Thus, we have
$$
S(k) = f(t_1, \dots, t_k, a)+\sum_{i=k+2}^{|V(G)|} t_i
$$
and so we may derive $\sum_{i=k+2}^{|V(G)|} t_i$. Thus, we have all the necessary information to compute
$$
a-\sum_{i=1}^k t_i-\sum_{i=k+2}^{|V(G)|} t_i = t_{k+1}.
$$

This process may be repeated to find all twig lengths.

\end{proof}

\section{Acknowledgments}

The author would like to thank Sophie Spirkl for helpful comments.

\bibliographystyle{plain}
\bibliography{bib}

\end{document}